\documentclass[11pt]{article}

\usepackage[T1]{fontenc}
\usepackage{lmodern}
\usepackage{amsmath,amssymb,amsthm,mathtools}
\usepackage{microtype}
\usepackage{tikz}
\usetikzlibrary{arrows.meta,calc,positioning}
\usepackage[margin=1in]{geometry}
\usepackage{enumitem}
\usepackage{booktabs}
\usepackage{array}
\usepackage{xcolor}
\usepackage{subcaption}
\usepackage[section]{placeins}
\usepackage[numbers,sort&compress]{natbib}
\usepackage[hidelinks]{hyperref}
\usepackage[nameinlink,capitalise,noabbrev]{cleveref}

\definecolor{oldred}{RGB}{165,48,48}
\definecolor{newblue}{RGB}{30,88,155}
\hypersetup{
  pdftitle={The Class Edge-Reconstruction Number of a Maximal Planar Graph Is One or Two},
  pdfauthor={Sergey Ivanov}
}

\newtheorem{theorem}{Theorem}[section]
\newtheorem{lemma}[theorem]{Lemma}
\newtheorem{proposition}[theorem]{Proposition}
\newtheorem{corollary}[theorem]{Corollary}

\theoremstyle{definition}

\theoremstyle{remark}
\newtheorem{remark}[theorem]{Remark}

\newcommand{\M}{\mathcal M}
\newcommand{\ED}{\operatorname{ED}}
\newcommand{\Cern}{\operatorname{Cern}}

\newcommand{\eps}{\boldsymbol\varepsilon}

\title{The Class Edge-Reconstruction Number of a\\
Maximal Planar Graph Is One or Two}
\author{Sergey Ivanov}
\date{}

\begin{document}
\maketitle

\begin{abstract}
An edge card of a graph is obtained by deleting one edge, and a class
edge-reconstruction number asks for the fewest carefully selected cards that
identify the graph when its class is known.  We determine the sharp universal
bound for maximal planar graphs.  Two selected cards always suffice, and the
octahedral graph shows that two can be necessary; some maximal planar graphs
are already identified by one card.  The argument exploits the fact that
deleting a flippable edge leaves a single quadrilateral whose two diagonals
give the only possible maximal-planar completions.  Degree information then
rules out the competing completion, with a separate argument for graphs
containing a vertex of degree three.  This settles a problem posed in a 2010
survey on reconstruction numbers.
\end{abstract}


\section{Introduction}

Graph reconstruction asks how much information about an unlabeled graph
survives after one vertex or one edge has been removed.  Harary formulated
the edge-reconstruction problem in the early development of the subject
\cite{Harary1964}; Bondy and Hemminger's classical survey records the broader
vertex and edge theories and their connections \cite{BondyHemminger1977}.
Harary and Plantholt then introduced the quantitative question relevant here:
how many \emph{carefully selected} cards, rather than the whole deck, are
actually needed to identify a graph \cite{HararyPlantholt1985}?  The survey of
Asciak, Francalanza, Lauri, and Myrvold gives the corresponding multiset
definition for class edge reconstruction and isolates the present question
as its Problem~2.5 \cite{AsciakEtAl2010}.  Lauri had earlier recommended the
class edge-reconstruction numbers of trees and maximal planar graphs as a
particularly interesting direction \cite{Lauri1993}. Recent results illustrate how much ambiguity can survive vertex deletion.
Nonisomorphic graphs can share an arbitrarily large fraction of their
vertex-deleted cards \cite{IvanovDeckOverlap2026}, while pseudosimilarity can
persist through consecutive deletions for all but a sublinear number of
vertices \cite{IvanovPseudosimilarity2026}.

Maximal planar graphs are a natural testing ground.  They are precisely the
simple triangulations of the sphere (apart from harmless small-order
conventions), and their embeddings are rigid.  On the vertex side, Fiorini
and Lauri proved recognition of maximal planarity from the full vertex deck
\cite{FioriniLauri1981Recognition}, and Lauri completed the reconstruction
theorem \cite{Lauri1981Reconstruction}.  Harary and Lauri obtained the much
stronger selected-card result: within the class of maximal planar graphs the
class \emph{vertex} reconstruction number is always one or two, with the
one-card graphs characterized \cite{HararyLauri1987}.  This result is the
closest predecessor of the theorem proved here, but it concerns
vertex-deleted cards.

There is also substantial full-deck edge theory.  Fiorini proved that maximal
planar graphs of minimum degree at least four are edge-reconstructible
\cite{Fiorini1978Edge}.  Lauri treated planar graphs of minimum degree five
\cite{Lauri1979}; Fiorini and Lauri proved full-edge-deck results for
four-connected planar graphs \cite{FioriniLauri1982FourConnected} and for
important classes of surface triangulations
\cite{FioriniLauri1982Surfaces}.  Fan's four-part series culminated in
edge-reconstructibility for planar graphs of minimum degree at least three
\cite{FanSeriesI,FanSeriesII,FanSeriesIII,FanSeriesIV}; Zhao obtained further surface-embedding results
\cite{Zhao1993,Zhao1998II,Zhao1998III}.  Greenwell's classical reduction from
the edge deck to the vertex deck supplies another bridge between the two
theories \cite{Greenwell1971}, and Maccari, Rueda, and Viazzi survey the
full-deck edge literature \cite{MaccariEtAl2002}.

Those theorems establish reconstruction from the \emph{complete} edge deck.
They do not bound the number of selected cards needed when the parent is
promised to be maximal planar.  This distinction is substantial: a maximal
planar graph on \(n\) vertices has \(3n-6\) edge cards, whereas the result
below reduces the witness to at most two.

Asciak et al. asked exactly:
\begin{quote}
``What is the class edge-reconstruction number of a maximal planar graph?''
\cite[Problem~2.5]{AsciakEtAl2010}
\end{quote}
We prove that the answer is always one or two and that two is the sharp
universal bound.

The central geometric observation is simple.  Deleting a flippable edge
merges its two incident triangular faces into a quadrilateral.  The two
diagonals of that quadrilateral give the only possible maximal-planar parents
of the card.  Thus the first card leaves at most one nonisomorphic rival, and
the task of the second card is merely to exclude that rival.  Degree
signatures do so when the minimum degree is at least four.  When a
degree-three vertex is present, a degree-transfer identity provides the
additional rigidity.

The 2010 survey later says that no maximal planar graph has class
edge-reconstruction number one \cite[p.~453]{AsciakEtAl2010}.  Under the
multiset definition printed in that survey, this ancillary sentence is not
correct: the unique five-vertex triangulation is already identified by one
edge card, and \cref{prop:infinite-one-card} gives infinitely many further
examples.  This correction does not diminish Problem~2.5 itself; it clarifies
one of the two values in its resolution.

\section{Selected edge cards and the main result}

All graphs in this paper are finite and simple.  Let \(G\) be a graph and
\(e\in E(G)\).  The graph \(G-e\), with both endpoints retained, is an
\emph{edge card}.  The edge deck
\[
  \ED(G)=\{\!\{,G-e:e\in E(G),\}\!\}
\]
is the multiset of its unlabeled edge-card isomorphism types.

Let \(\M\) be the class of maximal planar graphs.  A selected subdeck
\(S\preccurlyeq \ED(G)\) \emph{identifies \(G\) within \(\M\)} if, for every
\(H\in\M\),
\[
             S\preccurlyeq \ED(H)\quad\Longrightarrow\quad H\cong G.
\]
The least size of such an \(S\), counting multiplicity, is denoted by
\(\Cern_{\M}(G)\).

\begin{remark}[The multiset convention]\label{rem:multiset}
Two selected cards may be two distinct physical copies of the same
isomorphism type.  They may not exceed that type's multiplicity in the deck.
This convention is part of the definition in \cite{AsciakEtAl2010} and is
essential for the octahedral sharpness example in \cref{prop:octahedron}.
\end{remark}

\begin{theorem}[Main theorem]\label{thm:main}
Every finite simple maximal planar graph can be identified, among maximal
planar graphs, from at most two selected edge cards.  Some maximal planar
graphs require two cards.
\end{theorem}

\begin{corollary}\label{cor:values}
For every \(G\in\M\),
\[
                 \Cern_{\M}(G)\in\{1,2\},
\]
and the maximum possible value is \(2\).
\end{corollary}

There is also an exact, if deliberately card-based, test for which value
occurs.  A \emph{maximal-planar parent} of a card \(X\) is a graph
\(H\in\M\) for which \(H-f\cong X\) for some edge \(f\).

\begin{proposition}[One-card test]\label{prop:one-card-test}
A maximal planar graph \(G\) has class edge-reconstruction number one if and
only if one of its edge cards has no nonisomorphic maximal-planar parent.
\end{proposition}

This is an exact test, not an intrinsic structural classification of all
one-card triangulations.  Its short proof is in \cref{app:assembly}.

\section{The local diagonal picture}

Fix the spherical embedding of a maximal planar graph \(G\), and let
\(e=uv\).  The two facial triangles containing \(e\) have third vertices
\(x\) and \(y\).  The edge \(e\) is \emph{flippable} when \(xy\notin E(G)\).
Its diagonal flip is
\[
                    G^e=G-uv+xy.
\]

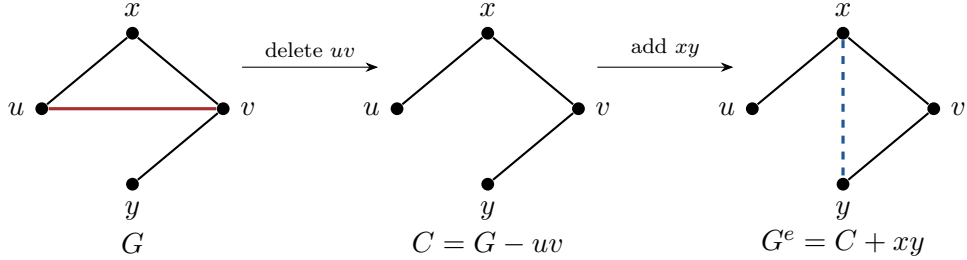
\begin{figure}[htbp]
\centering
\begin{tikzpicture}[
  vertex/.style={circle,fill=black,inner sep=1.7pt},
  boundary/.style={line width=.8pt},
  oldedge/.style={oldred,very thick},
  newedge/.style={newblue,very thick,dashed},
  >=Stealth
]
\begin{scope}[xshift=-4.7cm]
  \node[vertex,label=left:$u$]  (gU) at (-1.2,0) {};
  \node[vertex,label=right:$v$] (gV) at ( 1.2,0) {};
  \node[vertex,label=above:$x$] (gX) at (0,1) {};
  \node[vertex,label=below:$y$] (gY) at (0,-1) {};
  \draw[boundary] (gU)--(gX)--(gV)--(gY)--cycle;
  \draw[oldedge] (gU)--(gV);
  \node at (0,-1.78) {$G$};
\end{scope}

\begin{scope}
  \node[vertex,label=left:$u$]  (cU) at (-1.2,0) {};
  \node[vertex,label=right:$v$] (cV) at ( 1.2,0) {};
  \node[vertex,label=above:$x$] (cX) at (0,1) {};
  \node[vertex,label=below:$y$] (cY) at (0,-1) {};
  \draw[boundary] (cU)--(cX)--(cV)--(cY)--cycle;
  \node at (0,-1.78) {$C=G-uv$};
\end{scope}

\begin{scope}[xshift=4.7cm]
  \node[vertex,label=left:$u$]  (hU) at (-1.2,0) {};
  \node[vertex,label=right:$v$] (hV) at ( 1.2,0) {};
  \node[vertex,label=above:$x$] (hX) at (0,1) {};
  \node[vertex,label=below:$y$] (hY) at (0,-1) {};
  \draw[boundary] (hU)--(hX)--(hV)--(hY)--cycle;
  \draw[newedge] (hX)--(hY);
  \node at (0,-1.78) {$G^e=C+xy$};
\end{scope}

\draw[->] (-3.25,.55)--(-1.45,.55)
  node[midway,above,font=\scriptsize]{delete $uv$};
\draw[->] (1.45,.55)--(3.25,.55)
  node[midway,above,font=\scriptsize]{add $xy$};
\end{tikzpicture}
\caption{Deleting a flippable edge opens one quadrilateral.  Because the
resulting card is three-connected, its spherical embedding is unique.  The
original and flipped diagonals therefore give its only maximal-planar
parents.  Only the local part of the embedding is shown.}
\label{fig:flip}
\end{figure}

\begin{lemma}[Two-parent lemma]\label{lem:two-parent}
If \(e\) is a flippable edge of a maximal planar graph \(G\), then \(G-e\)
is three-connected.  Its only maximal-planar parents are \(G\) and the flip
mate \(G^e\), up to isomorphism.
\end{lemma}

The lemma turns reconstruction into a two-candidate problem.  If the two
parents in \cref{fig:flip} are isomorphic, the first card already identifies
\(G\).  Otherwise it remains only to choose a second card of \(G\) that is
absent from the deck of \(G^e\), or to use a multiplicity that \(G^e\) does
not have.

The proof, including the three-connectivity point that makes the statement
about parents rigorous, is in \cref{app:two-parent}.

\section{Road map of the proof}

Every maximal planar graph has minimum degree at least three, and Euler's
formula guarantees a vertex of degree at most five.  The proof therefore has
two parts.

\subsection{Minimum degree at least four}

Let \(\delta\) be the minimum degree.  Here \(\delta\) is four or five.  At a
minimum-degree vertex, an incident edge fails to be flippable precisely when
a chord between two consecutive-around-the-vertex neighbours blocks the
flip.  Those blocking chords lie in the disk complementary to the closed
star of the vertex and cannot cross.  A quadrilateral admits at most one
such diagonal, and a pentagon at most two.  Consequently a degree-four
vertex has at least two flippable incident edges, and a degree-five vertex
has at least three; see \cref{fig:blockers}.

\begin{figure}[htbp]
\centering
\begin{tikzpicture}[
  vtx/.style={circle,fill=black,inner sep=1.5pt},
  spoke/.style={black!70,line width=.65pt},
  cycle/.style={black,line width=.75pt},
  blocker/.style={oldred,very thick,dashed}
]
\begin{scope}[xshift=-2.8cm]
  \node[vtx,label=below:$v$] (v4) at (0,0) {};
  \foreach \i/\ang in {0/90,1/0,2/270,3/180}{
    \node[vtx] (q\i) at (\ang:1.25) {};
    \draw[spoke] (v4)--(q\i);
  }
  \draw[cycle] (q0)--(q1)--(q2)--(q3)--cycle;
  \draw[blocker] (q0) to[out=28,in=-28,looseness=1.65] (q2);
  \node[align=center,font=\small] at (0,-1.8)
    {degree four:\\one chord blocks two spokes};
\end{scope}
\begin{scope}[xshift=2.8cm]
  \node[vtx,label=below:$v$] (v5) at (0,0) {};
  \foreach \i/\ang in {0/90,1/18,2/-54,3/-126,4/162}{
    \node[vtx] (p\i) at (\ang:1.3) {};
    \draw[spoke] (v5)--(p\i);
  }
  \draw[cycle] (p0)--(p1)--(p2)--(p3)--(p4)--cycle;
  \draw[blocker] (p0) to[out=25,in=5,looseness=1.5] (p2);
  \draw[blocker] (p0) to[out=155,in=175,looseness=1.5] (p3);
  \node[align=center,font=\small] at (0,-1.8)
    {degree five:\\at most two noncrossing blockers};
\end{scope}
\end{tikzpicture}
\caption{The link cycle around a minimum-degree vertex.  A dashed chord
blocks the spoke whose two facial opposite vertices it joins.  Noncrossing
of the blocker chords leaves at least two flippable spokes at degree four and
at least three at degree five.}
\label{fig:blockers}
\end{figure}
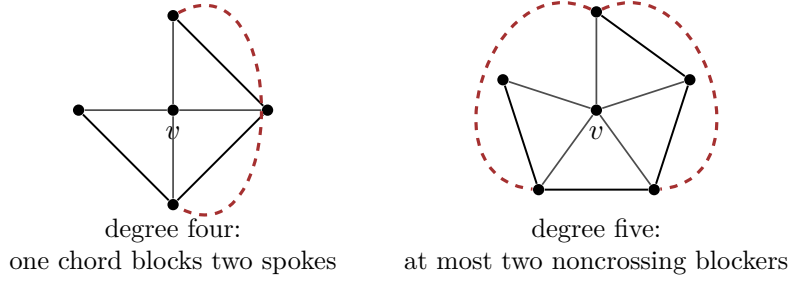

After a suitable flip, the old diagonal's endpoints lose one degree and the
new diagonal's endpoints gain one.  According to whether vertices above the
minimum degree contain an edge, form a nonempty independent set, or do not
exist, one chooses a second deletion whose low-degree signature cannot occur
in the flip mate.  In the regular case the two selected cards may be
isomorphic; multiplicity then excludes the rival.

\begin{proposition}[The noncubic case]\label{prop:high-degree}
A maximal planar graph of minimum degree at least four is identified by at
most two selected edge cards.
\end{proposition}

The complete three-case argument is in \cref{app:high-degree}.

\subsection{A vertex of degree three}

Let \(v\) have neighbours \(a,b,c\).  After the small graphs are removed,
the three edges \(ab,bc,ca\) are all flippable.  Suppose, for contradiction,
that a chosen boundary card and every possible second card also occur in its
flip mate.  Comparing the degree sequences before and after deletion forces
one fixed degree class to meet every edge of \(G\).  The possible residual
degree balances then have only four outcomes: regularity, bipartiteness, the
unique five-vertex triangulation, or a locally impossible neighbourhood.
Each conflicts with the standing counterexample.

The calculation is kept out of the main text.  \Cref{app:degree-bookkeeping}
gives the degree-transfer identity, and \cref{app:cubic} applies it to the
three edges surrounding \(v\).

\begin{proposition}[The cubic case]\label{prop:cubic}
A maximal planar graph containing a vertex of degree three is identified by
at most two selected edge cards.
\end{proposition}

\section{Sharpness and one-card reconstruction}

\subsection{The smallest transparent one-card example}

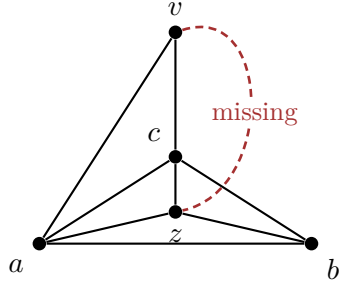
\begin{figure}[htbp]
\centering
\begin{tikzpicture}[
  vertex/.style={circle,fill=black,inner sep=1.8pt},
  edge/.style={line width=.8pt}
]
\node[vertex,label=above:$v$]       (v) at (0,1.8) {};
\node[vertex,label=below left:$a$]  (a) at (-1.8,-1) {};
\node[vertex,label=below right:$b$] (b) at (1.8,-1) {};
\node[vertex,label=above left:$c$]  (c) at (0,.15) {};
\node[vertex,label=below:$z$]       (z) at (0,-.58) {};

\draw[edge] (v)--(a)--(b)--cycle;
\draw[edge] (c)--(v);
\draw[edge] (c)--(a);
\draw[edge] (c)--(b);
\draw[edge] (z)--(a);
\draw[edge] (z)--(b);
\draw[edge] (z)--(c);
\draw[oldred,densely dashed,line width=1pt]
  (v) to[out=18,in=12,looseness=1.35] (z);
\node[oldred,fill=white,inner sep=1pt,font=\small] at (1.05,.72)
  {missing};
\end{tikzpicture}
\caption{The graph \(K_5-vz\), the unique maximal planar graph on five
vertices.  An edge card retains five vertices, so every maximal-planar parent
of that card is isomorphic to this graph.  One selected card therefore
suffices.  The dashed curve marks the sole nonedge and is not part of the
graph.}
\label{fig:one-card}
\end{figure}

\begin{proposition}\label{prop:k5minus}
The unique five-vertex maximal planar graph \(K_5-e\) has class
edge-reconstruction number one.
\end{proposition}

The example is not a small-order accident.  Let \(B_m\) be the bipyramid over
the rim cycle \(r_0r_1\dots r_{m-1}r_0\): two nonadjacent hubs \(p,q\) are
adjacent to every rim vertex.  Insert a new vertex \(w\) in the face
\(pr_0r_1\) and join it to the three corners, obtaining \(J_m\); see
\cref{fig:infinite-family}.

\begin{figure}[htbp]
\centering
\begin{tikzpicture}[
  vertex/.style={circle,fill=black,inner sep=1.6pt},
  edge/.style={line width=.75pt},
  oldedge/.style={oldred,very thick},
  missing/.style={oldred,very thick,dashed},
  newedge/.style={newblue,very thick,dotted},
  >=Stealth
]
% Only the neighbourhood of the rim edge r_0r_1 is shown in each panel.
\begin{scope}[xshift=-4.6cm]
  \node[vertex,label=left:$r_0$] (Lr0) at (-1,0) {};
  \node[vertex,label=right:$r_1$] (Lr1) at (1,0) {};
  \node[vertex,label=above:$p$] (Lp) at (0,1.35) {};
  \node[vertex,label=below:$q$] (Lq) at (0,-1.15) {};
  \node[vertex,label=right:$w$] (Lw) at (0,.52) {};
  \draw[edge] (Lr0)--(Lr1) (Lr0)--(Lp)--(Lr1)
              (Lr0)--(Lq)--(Lr1) (Lw)--(Lr0) (Lw)--(Lr1);
  \draw[oldedge] (Lw)--(Lp);
  \node at (0,-1.72) {$J_m=X+wp$};
\end{scope}
\begin{scope}
  \node[vertex,label=left:$r_0$] (Mr0) at (-1,0) {};
  \node[vertex,label=right:$r_1$] (Mr1) at (1,0) {};
  \node[vertex,label=above:$p$] (Mp) at (0,1.35) {};
  \node[vertex,label=below:$q$] (Mq) at (0,-1.15) {};
  \node[vertex,label=right:$w$] (Mw) at (0,.52) {};
  \draw[edge] (Mr0)--(Mr1) (Mr0)--(Mp)--(Mr1)
              (Mr0)--(Mq)--(Mr1) (Mw)--(Mr0) (Mw)--(Mr1);
  \draw[missing] (Mw)--(Mp);
  \node at (0,-1.72) {$X=J_m-wp$};
\end{scope}
\begin{scope}[xshift=4.6cm]
  \node[vertex,label=left:$r_0$] (Rr0) at (-1,0) {};
  \node[vertex,label=right:$r_1$] (Rr1) at (1,0) {};
  \node[vertex,label=above:$q$] (Rq) at (0,1.35) {};
  \node[vertex,label=below:$p$] (Rp) at (0,-1.15) {};
  \node[vertex,label=right:$w$] (Rw) at (0,.52) {};
  \draw[edge] (Rr0)--(Rr1) (Rr0)--(Rq)--(Rr1)
              (Rr0)--(Rp)--(Rr1) (Rw)--(Rr0) (Rw)--(Rr1);
  \draw[newedge] (Rw)--(Rq);
  \node at (0,-1.72) {$X+wq\cong J_m$};
\end{scope}
\draw[->] (-3.12,.65)--(-1.48,.65)
  node[midway,above,font=\scriptsize]{delete $wp$};
\draw[->] (1.48,.65)--(3.12,.65)
  node[midway,above,font=\scriptsize]{add $wq$};
\end{tikzpicture}
\caption{The local picture for the infinite one-card family; the rest of the
bipyramid is unchanged and is not shown.  In the middle card, \(w\) has
degree two.  A maximal-planar completion must join it to one of the two hubs.
The left and right completions are isomorphic under the hub swap
\(p\leftrightarrow q\), so the middle card identifies \(J_m\).  The dashed
red segment marks the deleted edge; the dotted blue segment marks the
alternative completion.}
\label{fig:infinite-family}
\end{figure}
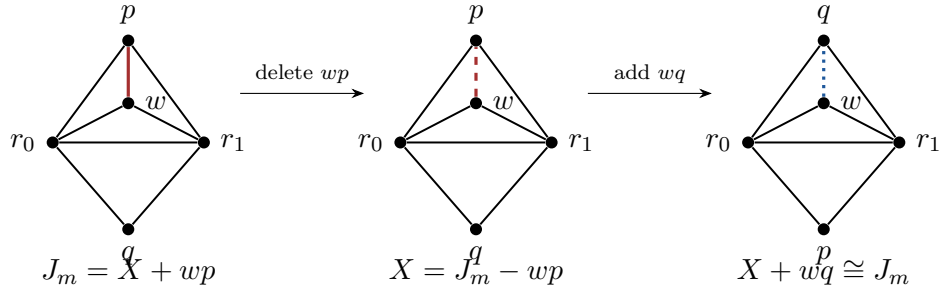

\begin{proposition}[Infinitely many one-card graphs]\label{prop:infinite-one-card}
For every \(m\ge4\), the maximal planar graph \(J_m\) has class
edge-reconstruction number one.
\end{proposition}

\FloatBarrier
\subsection{Why the bound two is sharp}

Let \(O=K_{2,2,2}\), the octahedral graph.  It is four-regular and
edge-transitive.  Deleting any edge gives the same card type \(C\).

\begin{figure}[htbp]
\centering
\begin{tikzpicture}[
  vertex/.style={circle,fill=black,inner sep=1.65pt},
  edge/.style={line width=.7pt},
  oldedge/.style={oldred,very thick},
  newedge/.style={newblue,very thick,dashed}
]
% Common planar embedding: outer triangle A B C, inner triangle a b c.
\begin{scope}[xshift=-3.3cm]
  \node[vertex,label=above:$A$]       (OA) at (0,1.8) {};
  \node[vertex,label=below left:$B$]  (OB) at (-1.7,-1) {};
  \node[vertex,label=below right:$C$] (OC) at (1.7,-1) {};
  \node[vertex,label=below:$a$]       (Oa) at (0,-.55) {};
  \node[vertex,label=right:$b$]       (Ob) at (.72,.25) {};
  \node[vertex,label=left:$c$]        (Oc) at (-.72,.25) {};
  \draw[edge] (OA)--(OB)--(OC)--cycle;
  \draw[edge] (Oa)--(Ob)--(Oc)--cycle;
  \draw[edge] (Oa)--(OB) (Oa)--(OC);
  \draw[edge] (Ob)--(OA) (Ob)--(OC);
  \draw[edge] (Oc)--(OA) (Oc)--(OB);
  \draw[oldedge] (OA)--(OB);
  \node at (0,-1.55) {octahedron (O)};
\end{scope}

\begin{scope}[xshift=3.3cm]
  % A straight-line planar embedding of the flip mate.
  \node[vertex,label=below left:$A$]  (HA) at (-1.75,-.9) {};
  \node[vertex,label={[xshift=3pt,yshift=-4pt]right:$B$}] (HB) at (.15,.1) {};
  \node[vertex,label=below right:$C$] (HC) at (1.85,-.9) {};
  \node[vertex,label={[xshift=3pt,yshift=2pt]right:$a$}]   (Ha) at (.15,.62) {};
  \node[vertex,label=above:$b$]       (Hb) at (.15,1.38) {};
  \node[vertex,label=above left:$c$]  (Hc) at (-.75,-.35) {};
  \draw[edge] (HA)--(HC) (HA)--(Hb) (HA)--(Hc);
  \draw[edge] (HB)--(HC) (HB)--(Ha) (HB)--(Hc);
  \draw[edge] (HC)--(Ha) (HC)--(Hb);
  \draw[edge] (Ha)--(Hb) (Ha)--(Hc) (Hb)--(Hc);
  \draw[newedge] (HC)--(Hc);
  \node at (0,-1.55) {flip mate $H$};
\end{scope}
\node[align=center,font=\small] at (0,-2.2)
  {$H=O-AB+Cc$ and $O-AB=H-Cc$};
\end{tikzpicture}
\caption{One card does not identify the octahedron.  The solid red edge
\(AB\) and dashed blue edge \(Cc\) are the two diagonals of the same opened
quadrilateral, so their deletions give the same card.  The octahedron has
degree sequence \(4^6\), whereas its flip mate has degree sequence
\(3^2,4^2,5^2\).  Only deletion of the flip mate's degree-five--degree-five
edge gives a card with degree sequence \(3^2,4^4\); hence the rival contains
only one copy of the common card.  Two copies identify the octahedron.}
\label{fig:octahedron}
\end{figure}
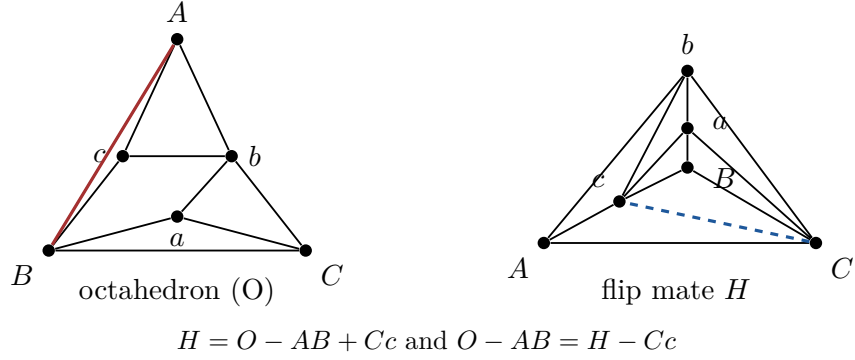

\begin{proposition}[Sharpness]\label{prop:octahedron}
The octahedral graph has class edge-reconstruction number two.
\end{proposition}

Thus both alternatives in \cref{cor:values} occur, and both occur for
nontrivial reasons: a one-card graph may have two possible completions that
are isomorphic, while a two-card graph may require multiplicity to distinguish
two nonisomorphic completions.

\FloatBarrier
\section{Concluding remarks}

The theorem resolves the numerical question but leaves a natural structural
problem: characterize intrinsically the maximal planar graphs for which one
card suffices.  \Cref{prop:one-card-test} is exact but expressed in terms of
parents of a card; \cref{prop:infinite-one-card} shows that the one-card side
is not confined to finitely many small exceptions.

The proof also illustrates why full-deck reconstructibility and
reconstruction numbers should be kept separate.  The older full-deck results
provide important context, but the two-card theorem comes from controlling
one local ambiguity and then finding a short certificate that destroys it.
This viewpoint may be useful for triangulations of other surfaces, where an
edge deletion again exposes local completions but embedding uniqueness and
global topology require additional care.

\clearpage
\nocite{*}
\bibliographystyle{unsrtnat}
\bibliography{maximal_planar_edge_reconstruction}

\clearpage
\appendix

\section{Planar preliminaries and small graphs}\label{app:preliminaries}

We use the standard spherical embedding of a maximal planar graph.  If
\(n\ge3\), every face is triangular, \(|E|=3n-6\), the graph is
three-connected when \(n\ge4\), and
\[
                  3\le \delta(G)\le5.
\]
The upper bound follows from the average degree
\(2|E|/n=6-12/n<6\).  Around a vertex \(v\), its neighbours occur on a
cycle, called the \emph{link} of \(v\).

\begin{lemma}[Adjacent cubic vertices]\label{lem:adjacent-cubic}
If two degree-three vertices of a simple maximal planar graph are adjacent,
then the graph is \(K_4\).
\end{lemma}

\begin{proof}
Let \(x,y\) be adjacent cubic vertices, and let \(r,s\) be the two facial
opposites of \(xy\).  Since both \(x\) and \(y\) have degree three, their
neighbour sets are \(\{y,r,s\}\) and \(\{x,r,s\}\), respectively.  The
link of \(x\) contains \(rs\), so \(x,y,r,s\) induce \(K_4\).  Moreover the
four triangles \(xyr,xys,xrs,yrs\) are faces of the given triangulation.
They are the four faces of the spherical embedding of \(K_4\), so there is
no face in which another vertex could lie.  Hence the whole graph is
\(K_4\).
\end{proof}

\begin{corollary}\label{cor:cubic-independent}
In a maximal planar graph other than \(K_4\), the cubic vertices form an
independent set.  If \(v\) is cubic with neighbours \(a,b,c\), then
\(G-v\) is maximal planar and \(abc\) is one of its faces.
\end{corollary}

\begin{proof}
Independence follows from \cref{lem:adjacent-cubic}.  The neighbours of
\(v\) occur as a three-cycle in its link.  Deleting \(v\) replaces the three
incident triangular faces by the single face \(abc\); every face remains
triangular, so \(G-v\) is maximal planar.
\end{proof}

\begin{lemma}[Small orders]\label{lem:small-orders}
The maximal planar graphs on three, four, and five vertices are unique up to
isomorphism, namely \(K_3\), \(K_4\), and \(K_5-e\).  Each has class
edge-reconstruction number one.
\end{lemma}

\begin{proof}
The uniqueness assertions follow directly from \(|E|=3n-6\).  For a fixed
order among three, four, and five, an edge card has that same number of
vertices, and any maximal-planar parent must be the unique graph of that
order.  Hence any one card identifies its parent within \(\M\).
\end{proof}

\section{The two-parent lemma}\label{app:two-parent}

\begin{proof}[Proof of \cref{lem:two-parent}]
Write \(e=uv\), let \(x,y\) be its facial opposite vertices, and put
\[
                  C=G-uv,\qquad H=C+xy=G^e.
\]
Both \(G\) and \(H\) are maximal planar and hence three-connected.

We first prove that \(C\) is three-connected.  Suppose that \(C-S\) is
disconnected for some \(|S|\le2\).  Since \(G-S=(C-S)+uv\) is connected,
the vertices \(u,v\) lie outside \(S\) and in different components of
\(C-S\).  Since \(H-S=(C-S)+xy\) is connected, the vertices \(x,y\) also
lie outside \(S\) and in different components.  But the path \(u-x-v\) is
present in \(C-S\), putting \(u\) and \(v\) in the same component, a
contradiction.

By Whitney's theorem \cite{Whitney1932}, \(C\) has a unique spherical
embedding up to reflection.  It has \(3n-7\) edges and hence, by Euler's
formula, \(2n-5\) faces.  The sum of its face lengths is \(6n-14\), exactly
one more than three times the number of faces.  Consequently it has one
quadrilateral face and every other face is triangular.  A one-edge planar
completion can only insert a diagonal of that quadrilateral.  Its two
diagonals are \(uv\) and \(xy\), giving \(G\) and \(H\).  These are
therefore the only maximal-planar parents of \(C\), up to isomorphism.
\end{proof}

\section{Degree bookkeeping}\label{app:degree-bookkeeping}

For a graph \(X\), let
\[
  \nu(X)=(\nu_0(X),\nu_1(X),\ldots),
\]
where \(\nu_i(X)\) is the number of vertices of degree \(i\).  Let
\(\eps_i\) be the \(i\)-th unit vector.  For an edge \(rs\in E(X)\), define
its endpoint-degree vector by
\[
             A_X(rs)=\eps_{d_X(r)}+\eps_{d_X(s)}.
\]
Finally let \(L\eps_i=\eps_{i-1}-\eps_i\), extended linearly to vectors of
finite support.

\begin{lemma}[Deletion identity]\label{lem:deletion-identity}
For every edge \(rs\) of \(X\),
\[
             \nu(X-rs)=\nu(X)+L A_X(rs).
\]
The map \(L\) is injective on finite-support integer vectors.
\end{lemma}

\begin{proof}
Deleting \(rs\) lowers the degrees of \(r\) and \(s\) by one, which gives
the displayed identity.  If a nonzero finite-support vector \(z\) has
largest nonzero coordinate \(k\), then the \(k\)-th coordinate of \(Lz\)
is \(-z_k\), and is nonzero.  Thus \(L\) is injective.
\end{proof}

\begin{lemma}[Residual identity]\label{lem:residual}
Suppose the same card has completions \(G=C+e\) and \(H=C+e'\), and set
\[
                  Q=A_G(e)-A_H(e').
\]
If \(G-g\cong H-h\), then
\[
                  A_G(g)-A_H(h)=Q.                 \tag{C.1}\label{eq:residual}
\]
Write \(Q=P-N\), where \(P,N\) have disjoint supports and nonnegative
integer entries.  Their common total mass \(r\) is zero, one, or two.  If
every edge card of \(G\) has a degree-sequence match among the edge cards of
\(H\), then:
\begin{enumerate}[label=\textup{(\roman*)},leftmargin=2.2em]
\item if \(r=2\), every edge of \(G\) has the same unordered endpoint-degree
pair;
\item if \(r=1\), the vertices of one fixed degree form a vertex cover of
\(G\);
\item if \(r=0\), then \(Q=0\).
\end{enumerate}
\end{lemma}

\begin{proof}
Apply \cref{lem:deletion-identity} to the isomorphic cards \(G-g\) and
\(H-h\), and also to the common card \(G-e=H-e'\).  Subtracting the two
equalities and using injectivity of \(L\) gives \eqref{eq:residual}.

The vectors \(A_G(e)\) and \(A_H(e')\) both have total mass two, so the
positive and negative parts of their difference have a common mass
\(r\in\{0,1,2\}\).  From \eqref{eq:residual},
\[
                  A_H(h)=A_G(g)-P+N.
\]
On the support of \(P\), nonnegativity forces \(P\le A_G(g)\)
coordinatewise.  If \(r=2\), the two vectors have the same total mass and
hence \(P=A_G(g)\) for every \(g\).  If \(r=1\), the single positive unit
must occur at an endpoint of every edge.  The case \(r=0\) is immediate.
\end{proof}

\section{Graphs of minimum degree at least four}\label{app:high-degree}

\begin{lemma}[Incident flips]\label{lem:incident-flips}
A degree-four vertex is incident with at least two flippable edges.  A
degree-five vertex is incident with at least three.
\end{lemma}

\begin{proof}
Let \(v_0,\ldots,v_{d-1}\) be the cyclic neighbours of \(v\), with indices
modulo \(d\).  The spoke \(vv_i\) is nonflippable exactly when the blocker
\(v_{i-1}v_{i+1}\) is present.  The closed star of \(v\) is a disk, so all
blockers are chords in the complementary disk and are pairwise noncrossing.
For \(d=4\), the two possible blocker diagonals each block two opposite
spokes, and at most one of the two diagonals can be present.  At least two
spokes are therefore flippable.  For \(d=5\), the five blockers are distinct
and a noncrossing set of pentagon diagonals has size at most two.  At least
three spokes are flippable.
\end{proof}

\begin{lemma}[A low--low flip]\label{lem:low-low-flip}
Let \(G\) be maximal planar with minimum degree \(\delta\in\{4,5\}\).  Put
\[
 L=\{v:d(v)=\delta\},\qquad R=V(G)\setminus L.
\]
If \(R\) is nonempty and independent, then \(G\) has a flippable edge with
both endpoints in \(L\).
\end{lemma}

\begin{proof}
Suppose instead that every \(L\)-\(L\) edge is nonflippable, and fix
\(v\in L\).  Since consecutive vertices of the link are adjacent, no two
\(R\)-vertices are consecutive there.

If \(\delta=5\), at most two link vertices lie in \(R\), so at least three
spokes from \(v\) end in \(L\).  Their nonflippability would require three
distinct noncrossing pentagon diagonals, contradicting
\cref{lem:incident-flips}.

Let \(\delta=4\).  If the link contains at most one vertex of \(R\), at
least three \(L\)-spokes are nonflippable.  The two possible quadrilateral
diagonals block opposite pairs of spokes, so blocking three requires both
crossing diagonals, impossible.  If the link contains two vertices of
\(R\), they occupy opposite positions.  The two \(L\)-spokes are also
opposite and share as blocker the diagonal joining those two \(R\)-vertices.
That blocker contradicts the independence of \(R\).  In every case we obtain
a contradiction.
\end{proof}

\begin{proof}[Proof of \cref{prop:high-degree}]
Euler's formula gives \(\delta=4\) or \(5\).  Use the sets \(L,R\) from
\cref{lem:low-low-flip}.

\smallskip
\noindent\emph{Case 1: \(G[R]\) contains an edge.}
Choose \(g\in E(G[R])\), choose \(v\in L\), and by
\cref{lem:incident-flips} choose a flippable edge \(e\) incident with
\(v\).  Let \(H=G^e\).  At least one old endpoint of \(e\) has degree
\(\delta-1\) in \(H\), and no vertex has smaller degree.  Every card of
\(H\) either retains such a vertex or lowers it to degree \(\delta-2\).
By contrast, both endpoints of \(g\) have degree at least \(\delta+1\) in
\(G\), so \(G-g\) has minimum degree \(\delta\).  Thus \(G-g\) is absent
from \(\ED(H)\).  The first card \(G-e\) has only the parents \(G,H\) by
\cref{lem:two-parent}; the pair \(G-e,G-g\) therefore identifies \(G\).

\smallskip
\noindent\emph{Case 2: \(R\) is nonempty and independent.}
By \cref{lem:low-low-flip}, choose a flippable \(L\)-\(L\) edge \(e=uv\),
and let \(H=G^e\).  The vertices \(u,v\) have degree \(\delta-1\) in
\(H\), and they are nonadjacent there.  Since \(G\) is connected and
\(R\ne\varnothing\), there is an \(L\)-\(R\) edge \(g\).  The card
\(G-g\) has exactly one degree-\(\delta-1\) vertex and none of smaller
degree.  A card of \(H\) that avoids \(u,v\) retains both
degree-\(\delta-1\) vertices; a deletion incident with one of them creates a
degree-\(\delta-2\) vertex and retains the other.  Hence no card of \(H\)
has the degree signature of \(G-g\).  Again \(G-e,G-g\) identify \(G\).

\smallskip
\noindent\emph{Case 3: \(R=\varnothing\).}
Now \(G\) is \(\delta\)-regular.  Choose a flippable edge \(e=uv\), with
facial opposites \(x,y\), and let \(H=G^e\).  In \(H\), the vertices
\(u,v\) have degree \(\delta-1\), the vertices \(x,y\) have degree
\(\delta+1\), and every other vertex has degree \(\delta\).  Every physical
card of \(G\) has degree signature
\[
               (\delta-1)^2\,\delta^{,n-2}.
\]
The only physical card of \(H\) with that signature is \(H-xy\): any other
deletion leaves at least one degree-\(\delta+1\) vertex.  Choose any edge
\(g\ne e\).  The selected two-card multiset \(\{\!\{G-e,G-g\}\!\}\) has
two cards with the regular-card signature, whereas \(H\) has only one.  It
therefore excludes \(H\), whether or not the two selected cards are
isomorphic.  The two-parent lemma completes the proof.
\end{proof}

\section{Graphs with a cubic vertex}\label{app:cubic}

\begin{proof}[Proof of \cref{prop:cubic}]
The graphs of order at most five are covered by
\cref{lem:small-orders}, so assume \(n\ge6\).  Suppose for a contradiction
that \(\Cern_{\M}(G)>2\).  Let \(v\) be cubic, with cyclic neighbourhood
\(N(v)=\{a,b,c\}\).

For each \(f\in\{ab,bc,ca\}\), let \(z_f\) be the third vertex of the face
on the side of \(f\) opposite \(v\).  The vertex \(z_f\) is outside
\(\{v,a,b,c\}\): otherwise those four vertices form \(K_4\), whose faces
and the degree of \(v\) leave no room for the remaining vertices.  Since
\(v\) has only the neighbours \(a,b,c\), the edge \(vz_f\) is absent.
Thus every \(f\) is flippable.  Put
\[
             C_f=G-f,\qquad H_f=G-f+vz_f.
\]
By \cref{lem:two-parent}, \(G\) and \(H_f\) are the only parents of
\(C_f\).  If \(H_f\cong G\), then \(C_f\) already identifies \(G\),
contrary to our assumption.  Hence \(H_f\not\cong G\).

Fix a boundary edge \(f\).  For every edge \(g\ne f\), the selected pair
\(C_f,G-g\) does not identify \(G\).  Its only possible nonisomorphic parent
after the first card is \(H_f\), so \(H_f\) contains a card isomorphic to
\(G-g\); if \(G-g\cong C_f\), the same conclusion holds with the required
multiplicity.  For \(g=f\), the common card itself is
\(G-f=H_f-vz_f\).  Consequently, for every edge \(g\) of \(G\), there is an
edge \(h\) of \(H_f\) such that
\[
                         G-g\cong H_f-h.           \tag{E.1}\label{eq:allmatch}
\]

For \(f=xy\), write \(s=d_G(z_f)\).  Since deleting \(xy\) lowers its two
endpoint degrees, while the competing completion inserts \(vz_f\), the
residual in \cref{lem:residual} is
\[
 Q_f=\eps_{d(x)}+\eps_{d(y)}-\eps_4-\eps_{s+1}.    \tag{E.2}\label{eq:cubicQ}
\]
After cancelling equal positive and negative units, let \(r_f\) be the mass
of the positive part.

We first show \(r_f\ne2\).  By \eqref{eq:allmatch} and
\cref{lem:residual}, mass two would force every edge of \(G\) to have one
fixed unordered endpoint-degree pair.  If the two degrees differ, every edge
joins the two corresponding degree classes and \(G\) is bipartite,
contradicting its triangular faces.  If they agree, connectedness makes
\(G\) regular; because \(G\) has a cubic vertex, Euler's formula then gives
\(G=K_4\), contrary to \(n\ge6\).  Hence \(r_f\in\{0,1\}\) for all three
boundary edges.

Suppose first that \(r_f=0\) for all \(f\in\{ab,bc,ca\}\).  Equation
\eqref{eq:cubicQ} says that each of the pairs
\[
           \{d(a),d(b)\},\quad \{d(b),d(c)\},\quad \{d(c),d(a)\}
\]
contains \(4\).  At least two of \(a,b,c\) therefore have degree four.
In the maximal planar graph \(T=G-v\), supplied by
\cref{cor:cubic-independent}, those two vertices are adjacent cubic
vertices.  \Cref{lem:adjacent-cubic} gives \(T=K_4\), hence \(n=5\), a
contradiction.

We may choose a boundary edge \(f\) with \(r_f=1\), and write its reduced
residual as
\[
                         Q_f=\eps_p-\eps_q.          \tag{E.3}\label{eq:massone}
\]
By \eqref{eq:allmatch} and \cref{lem:residual}, the degree-\(p\) vertices
meet every edge of \(G\).  We cannot have \(p=3\): by
\cref{cor:cubic-independent}, the cubic vertices would then be an independent
vertex cover, making both it and its complement independent and making
\(G\) bipartite.  Since \(v\) is not a degree-\(p\) vertex, each of the
edges \(va,vb,vc\) forces
\[
                         d(a)=d(b)=d(c)=p.           \tag{E.4}\label{eq:pneighbors}
\]
If \(p=4\), then \(G-v\) again has adjacent cubic vertices and is \(K_4\),
contrary to \(n\ge6\).  Thus \(p\ge5\).

Restart the construction for each boundary edge \(xy\), using its own flip
mate \(H_{xy}\), and apply \eqref{eq:cubicQ} separately.  By
\eqref{eq:pneighbors}, its uncancelled positive part is \(2\eps_p\), while
the negative part is \(\eps_4+\eps_{d(z_{xy})+1}\).  Since \(p\ne4\) and
residual mass two is impossible, one negative unit must cancel a copy of
\(\eps_p\).  Therefore
\[
       d(z_{xy})=p-1,
       \qquad Q_{xy}=\eps_p-\eps_4                  \tag{E.5}\label{eq:allQ}
\]
for each of \(ab,bc,ca\).

Let \(p=5\), and inspect the cyclic link of \(a\).  It contains the
consecutive segment \(b,v,c\).  On the complementary \(c\)-to-\(b\) arc,
the first and last vertices are \(z_{ac}\) and \(z_{ab}\), both of degree
four by \eqref{eq:allQ}.  The degree-five vertices cover every edge, so two
non-degree-five vertices cannot be adjacent.  If
\(z_{ac}=z_{ab}\), then the link has length four; if they are distinct,
they require at least one intermediate vertex and the link has length at
least six.  Both alternatives contradict \(d(a)=5\).

It remains that \(p\ge6\).  Fix \(f=ab\), write \(z=z_{ab}\), and put
\(H=H_f\).  The card \(G-va\) must occur in \(H\), so
\cref{lem:residual} and \eqref{eq:allQ} require an edge \(h\in E(H)\) with
\[
  A_H(h)=A_G(va)-Q_f
        =(\eps_3+\eps_p)-(\eps_p-\eps_4)
        =\eps_3+\eps_4.                              \tag{E.6}\label{eq:34edge}
\]
Thus \(H\) would contain a degree-three--degree-four edge.  It does not.
Indeed, the degree-\(p\) vertices cover every edge of \(G\), while the flip
changes only these degrees:
\[
  a,b:p\mapsto p-1,\qquad v:3\mapsto4,
  \qquad z:p-1\mapsto p.
\]
Every degree-three vertex of \(H\) is therefore an unchanged cubic vertex of
\(G\), and all its neighbours in \(H\) have degree at least
\(p-1\ge5\).  The new degree-four vertex \(v\) has neighbours of degrees
\(p-1,p-1,p,p\).  No degree-three--degree-four edge exists, contradicting
\eqref{eq:34edge}.  This final contradiction proves the proposition.
\end{proof}

\section{Sharpness, examples, and assembly}\label{app:assembly}

\begin{proof}[Proof of \cref{prop:one-card-test}]
If some card \(X\) of \(G\) has no nonisomorphic maximal-planar parent, then
the one-card subdeck \(\{\!\{X\}\!\}\) identifies \(G\) within \(\M\).
Conversely, if one card identifies \(G\), any maximal-planar parent of that
card must be isomorphic to \(G\).  This is exactly the stated condition.
\end{proof}

\begin{proof}[Proof of \cref{prop:k5minus}]
This is the order-five case of \cref{lem:small-orders}: every maximal planar
parent of a five-vertex edge card has five vertices, and \(K_5-e\) is the
unique maximal planar graph of that order.
\end{proof}

\begin{proof}[Proof of \cref{prop:infinite-one-card}]
Consider the card \(X=J_m-wp\).  The vertex \(w\) has degree two in \(X\),
with neighbours \(r_0,r_1\), while \(X-w=B_m\).  Every maximal planar graph
of order at least four has minimum degree at least three.  Hence any
maximal-planar one-edge completion of \(X\) must add an edge incident with
\(w\).  In such a completion \(w\) is cubic, and its three neighbours form a
facial triangle.  The common neighbours of the adjacent rim vertices
\(r_0,r_1\) in \(B_m\) are exactly the hubs \(p,q\) when \(m\ge4\).
Therefore the only completions are \(X+wp\) and \(X+wq\).  The automorphism
of \(B_m\) interchanging \(p\) and \(q\), and fixing the rim and \(w\), is
an isomorphism between these completions.  Thus every maximal-planar parent
of \(X\) is isomorphic to \(J_m\), and \cref{prop:one-card-test} gives
\(\Cern_{\M}(J_m)=1\).
\end{proof}

\begin{proof}[Proof of \cref{prop:octahedron}]
Let \(O=K_{2,2,2}\), and delete any edge \(e\).  Its two facial opposite
vertices are the two vertices in the third part and are nonadjacent, so
\(e\) is flippable.  The two parents of \(C=O-e\) are \(O\) and a
nonisomorphic flip mate \(H\), whose degree sequence is
\(3^2,4^2,5^2\).  Thus one card does not identify \(O\).

The octahedron is edge-transitive, so all twelve physical cards in
\(\ED(O)\) have the same type \(C\), with degree sequence \(3^2,4^4\).  In
\(H\), the only edge deletion with this degree sequence is deletion of the
new diagonal joining the two degree-five vertices.  Hence \(H\) contains
exactly one copy of \(C\).  By \cref{lem:two-parent}, any parent of one copy
of \(C\) is \(O\) or \(H\); a selected subdeck containing two copies excludes
\(H\).  Therefore \(\Cern_{\M}(O)=2\).
\end{proof}

\begin{proof}[Proof of \cref{thm:main,cor:values}]
By \cref{lem:small-orders}, the graphs of order at most five need one card.
For every larger maximal planar graph, the minimum degree is either three or
at least four.  \Cref{prop:cubic} handles the first case and
\cref{prop:high-degree} the second, proving the universal upper bound two.
The reconstruction number is at least one by definition.  The octahedron
requires two cards by \cref{prop:octahedron}, while \(K_5-e\) and the graphs
\(J_m\) require one by \cref{prop:k5minus,prop:infinite-one-card}.  Both
values occur and the class-wide maximum is two.
\end{proof}

\section{Independent computational audit}\label{app:computation}

The proof above is computer-free.  As an independent error check, all
unlabeled maximal planar graphs through thirteen vertices were generated with
\textsc{plantri}~5.5 \cite{BrinkmannMcKay2007}.  A separately written program
canonically encoded every edge card using \textsc{nauty} through its Python
interface \cite{McKayPiperno2014}, retained card multiplicities, and tested
every one- and two-card submultiset against every maximal-planar parent of the
same order.

\begin{table}[ht]
\centering
\caption{Exact exhaustive counts from the independent audit.  The two right
columns partition all unlabeled maximal planar graphs of each order.  No
graph with class edge-reconstruction number above two occurred.}
\label{tab:enumeration}
\begin{tabular}{@{}rrrr@{}}
\toprule
vertices & all graphs & number one & number two\\
\midrule
3  & 1    & 1    & 0\\
4  & 1    & 1    & 0\\
5  & 1    & 1    & 0\\
6  & 2    & 1    & 1\\
7  & 5    & 4    & 1\\
8  & 14   & 9    & 5\\
9  & 50   & 35   & 15\\
10 & 233  & 122  & 111\\
11 & 1249 & 444  & 805\\
12 & 7595 & 1523 & 6072\\
13 & 49566 & 5213 & 44353\\
\bottomrule
\end{tabular}
\end{table}

The audit also checked every flippable edge: at order thirteen alone there are
1,079,731 such edge occurrences.  In every case the deleted card had exactly
the original graph and its diagonal flip as parent isomorphism classes.  It
then checked the explicit witness strategies in \cref{app:high-degree} and
the cubic-boundary alternatives in \cref{app:cubic}; there were no failures.
These computations are corroborative only and are not used in any proof.
For reproducibility, the order-thirteen run used \texttt{pynauty}~2.8.8.1
and NetworkX~3.6.1.  The SHA-256 digests were as follows.
\begin{itemize}[leftmargin=1.5em,itemsep=1pt,topsep=3pt]
\item locally built \textsc{plantri} binary:\par
  {\scriptsize\ttfamily
  30d4736e0f5a4c48b91c1d57aa07b5d33574f751a896d732848f92886cf74793}
\item generated order-thirteen graph stream:\par
  {\scriptsize\ttfamily
  1a5fcd0741b8f118e291ffc3dd6d7153af670f831c75cc60480ede6aa69e4258}
\item independent verifier:\par
  {\scriptsize\ttfamily
  4dfb307b5b08de57de516d12f498fa742edf325ede7e8b5be87ebc1814ec03f3}
\item resulting JSON record:\par
  {\scriptsize\ttfamily
  3a91d4a4e50143642ea18510f6acaa1cc0fac4838c5a80f9edc411f83844d9dc}
\end{itemize}

\end{document}